\documentclass{amsart}
\usepackage[normalem]{ulem}
\usepackage{amsthm,amsbsy,amstext}
\usepackage{mathtools}
\usepackage{thmtools}
\declaretheoremstyle[headfont=\normalfont]{normalhead}
\usepackage{color}
\usepackage{graphicx}
\usepackage{morefloats}
\usepackage{hyperref}
\usepackage{amssymb}
\usepackage{textcomp}
\usepackage[autostyle]{csquotes}
\usepackage{xcolor}
\usepackage{todonotes}
\usepackage{comment}
\usepackage{url}
\usepackage{float}

\newtheoremstyle{mydef}
{\topsep}{\topsep}%
{}{}%
{\itshape}{}
{\newline}
{%
  \rule{\textwidth}{0.0pt}\\*%
  \thmname{#1}~\thmnumber{#2}\thmnote{\-\ #3}.\\*[-1.5ex]%
  \rule{\textwidth}{0.0pt}}%

\newtheorem{theorem}{Theorem}

\newtheorem{remark}{Remark}
\newtheorem{proposition}{Proposition}

\newtheorem{corollary}{Corollary}

\newtheorem{definition}{Definition}
\newcommand{\dimb}{\dim_B}
\newcommand{\udimb}{\overline{\text{dim}}_B}
\newcommand{\ldimb}{\underline{\text{dim}}_B}

\begin{document}

\title[]{Fractal Analysis of Hyperbolic Saddles\\ with Applications}
\author[V. Crnkovi\' c]{Vlatko Crnkovi\' c}
\address{University of Zagreb, Faculty of Electrical Engineering and Computing, Department of Applied Mathematics, Unska 3, 10000 Zagreb, Croatia; Hasselt University, Campus Diepenbeek, Agoralaan Gebouw D, 3590 Diepenbeek, Belgium}
\email{{vlatko.crnkovic@fer.hr}}

\author[R. Huzak]{Renato Huzak}
\address{Hasselt University, Campus Diepenbeek, Agoralaan Gebouw D, 3590 Diepenbeek, Belgium}
\email{{renato.huzak@uhasselt.be}}

\author[M. Resman]{Maja Resman}
\address{University of Zagreb, Faculty of Science, Department of Mathematics, Horvatovac 102a, 10000 Zagreb, Croatia}
\email{{maja.resman@math.hr}}

\date{May 31, 2023}


\begin{abstract}
In this paper we express the Minkowski dimension of spiral trajectories near hyperbolic saddles and semi-hyperbolic singularities in terms of the Minkowski dimension of intersections of such spirals with transversals near these singularities. We apply these results to hyperbolic saddle-loops and hyperbolic $2$-cycles to obtain upper bounds on the cyclicity of such limit periodic sets.
\end{abstract}

\maketitle

\section{Introduction}\label{sec1}

    The Minkowski dimension is a fractal dimension that quantifies how the Lebesgue measure of the $\delta$-neighborhood of a bounded set in $\mathbb{R}^N$ behaves as $\delta \to 0$.
    There are several equivalent ways of calculating this dimension but we mostly use the following one:
    
    \begin{definition}[\cite{falconer90}]
        Let $G$ be a bounded set in $N$-dimensional Euclidean space $\mathbb{R}^N$.
        Let
        \[ G_\delta \colon = \{ p \in \mathbb{R}^N \colon \text{dist}(p, G) < \delta\}, \delta > 0, \]
        be the $\delta$-neighborhood of $G$ and let $|G_\delta|$ be its Lebesgue measure.\\
        The upper and the lower Minkowski dimension of set $G$ are defined as the limits
        \[ \udimb\, G = \limsup_{\delta \to 0} \left[ N - \frac{ \ln |G_\delta|}{\ln \delta} \right]\quad \text{and}\quad \ldimb\, G = \liminf_{\delta \to 0} \left[ N - \frac{ \ln |G_\delta|}{\ln \delta} \right] \]
        respectively.
        If these two values are equal, the common value is called the Minkowski dimension of set $G$ and denoted by $\dimb\, G$.
    \end{definition}
    The Minkowski dimensions are preserved under bi-Lipschitz transformations, even when the image and the original are not in the same ambient space. More precisely, a transformation $\Psi: A \subset \mathbb{R}^N \to \mathbb{R}^K$ is called bi-Lipschitz if there exist positive constants $m$ and $M$ such that{\color{red},} for any $x,y \in A$, 
    \[ m || \Psi(x) - \Psi(y) || \leq ||x-y|| \leq M || \Psi(x) - \Psi(y)||. \]
    For a bi-Lipschitz map $\Psi: A \subset \mathbb{R}^N \to \mathbb{R}^K$ we have that
    \[  \ldimb A = \ldimb \Psi(A)\quad \text{and}\quad \udimb{A} = \udimb \Psi(A). \]
    All three Minkowski dimensions are monotone in the sense that, for $G\subseteq H, \dim G \leq \dim H$, when both are defined.
    In addition, the Minkowski dimension and the upper Minkowski dimension are finitely stable, meaning that $\dim\, (G \cup H)= \max \{ \dim\, G, \dim\, H\}$.
    For more on these and other properties of the Minkowski dimension we refer the reader to \cite{falconer90, tricot95}.
    \vspace{3mm}
    
    It is already known that the Minkowski dimension of spirals around weak foci and limit cycles of planar analytic vector fields yields information on the cyclicity of those limit periodic sets.
    First results of this type were obtained in \cite{zubzup05}.
    We state two main theorems that describe such connections.
    \begin{theorem}[Weak focus case, \cite{zubzup05, zubzup08}]
        Let $\Gamma$ be a spiral trajectory of the system
        \[ \begin{cases}
        \dot{r} = r(r^{2l} + \sum_{i=0}^{l-1} a_ir^{2i})\\
        \dot{\phi} = 1 \end{cases} \]
        near the origin.
        Then
        \begin{enumerate}
            \item[$(a)$] if $a_0 \neq 0$, then $\dimb \Gamma = 1$.
            \item[$(b)$] if $a_0 = a_1 = ... = a_{k-1} = 0,\ a_k \neq 0,\ k\geq 1$, then $\dimb \Gamma = \frac{4k}{2k+1}$.
        \end{enumerate}
    \end{theorem}
    \begin{theorem}[Limit cycle case, \cite{zubzup05, zubzup08}]
        Let the system
        \[ \begin{cases}
        \dot{r} = r(r^{2l} + \sum_{i=0}^{l-1} a_ir^{2i})\\
        \dot{\phi} = 1 \end{cases} \]
        have a limit cycle $r = a$ of multiplicity $m,\ 1\leq m \leq l$. Let $\Gamma_1$ and $\Gamma_2$ be spiral trajectories of this system near the limit cycle from outside or inside respectively.
        Then $\dimb \Gamma_1 = \dimb \Gamma_2 = 2 - \frac{1}{m}$.
    \end{theorem}
    Due to the Flow-Box Theorem (see for instance \cite[Theorem 1.12]{dumllibart06}), in order to calculate the dimension of spiral trajectories near limit cycles, it is sufficient to calculate the dimension of a sequence of points obtained by intersecting any such spiral with a transversal to the limit cycle (i.e. of the orbit of the first-return map on the transversal). For the case of foci, there is a variant of the Flow-Box Theorem developped in \cite{zubzup08}, called \emph{Flow-Sector Theorem}, that allows similar relation. For more details, see \cite{zubzup08}.
    
    In addition, due to results from \cite{neveda07}, there is a direct correspondence between the multiplicity of a fixed point of a line diffeomorphism (the first return map) and the Minkowski dimension of its orbit converging to the fixed point and, as a consequence, with the cyclicity of a focus/limit cycle.
    \vspace{3mm}
    
    In this paper we deal with spiral trajectories near more complex limit periodic sets: polycycles containing saddles and/or saddle-nodes.
    For a hyperbolic saddle with eigenvalues $\lambda_- < 0$ and $\lambda_+ > 0$, the hyperbolicity ratio is the quantity $r = -\frac{\lambda_-}{\lambda_+} > 0$.
    We read an upper bound on cyclicity of those sets in some known cases from the box dimension of their spiral trajectories.
    A better understanding of the cyclicity of such limit periodic sets is crucial for tackling the Hilbert's $16$th problem.
    It is not possible to use the Flow-Box Theorem to calculate the Minkowski dimension of spiral trajectories accumulating on such polycycles (from within), because the theorem does not apply near singularities, so we need a new method to calculate the Minkowski dimension of parts of the trajectories near singular points.
    Our main results are stated in Theorem \ref{thm1} and Theorem \ref{thm2} of Section \ref{sec2} which deal with neighborhoods of a hyperbolic saddle and a semi-hyperbolic singularity respectively.

    In Section \ref{sec3} we apply Theorem \ref{thm1} to a saddle-loop and find a relation between the codimension of the saddle-loop (an upper bound on the cyclicity of the loop) and the Minkowski dimension of its spiral trajectories.
    The Minkowski dimension depends only on the codimension of the loop, but the correspondence is $2$--$1$. For a more precise formulation, see Theorem \ref{thm3}.

    Finally, in Section \ref{sec4}, we apply Theorem \ref{thm1} to a hyperbolic $2$-cycle, and compare to cyclicity results obtained in \cite{mourtada94}.
    To summarize, for a non-resonant hyperbolic $2$-cycle with ratios of hyperbolicity $r_1 < 1 < r_2$ such that $r_1r_2 = 1$ and $r_1, r_2\not\in\mathbb{Q}$, the cyclicity of the $2$-cycle is shown not to be greater than $3 + (1+r_1)\frac{d-1}{2-d}$, where $d$ is the Minkowski dimension of any spiral trajectory near the $2$-cycle.
\medskip

    In the sequel we use two notions for the asymptotic behavior of functions as $x\to 0$.
    For $f(x)$ and $g(x)$ two positive functions with $x \approx 0$ and $x > 0$, we write $$f(x) \simeq g(x),\ x \to 0,$$ if there exist two positive constants $m$ and $M$ such that $mg(x) \leq f(x) \leq Mg(x)$ for all $x$ sufficiently small.
   For $f(x)$ and $g(x)$ two positive functions with $x \approx 0$ and $x > 0$, we write $$f(x) \sim g(x),\ x \to 0,$$ if
    \[ \lim_{x \to 0+} \frac{f(x)}{g(x)} = 1. \]

\section{The main results}\label{sec2}

    Let us explain the basic idea behind our method for calculating the Minkowski dimension of spiral trajectories of a polycycle.
    Due to the finite stability of the Minkowski dimension, in order to compute the dimension of a spiral trajectory, we consider separately different parts of the spiral: parts near the singular points and parts near the regular sides of the polycycle.
    The Minkowski dimension of the entire spiral is the maximal dimension of its constituting parts.

    The Flow-Box Theorem allows us to calculate the dimension of parts near the regular sides of the polycycle, but we need a new tool to calculate the dimension of the remaining parts.
    For any transversal to a regular side of a polycycle, the points of intersection of the spiral with the transversal define a sequence $(y_n)_n$.
    The distance between consecutive points of the sequence eventually starts to decrease.
    We take two transversals, one on each side of the singular point and sufficiently close to the singular point, in the domain where the saddle can be brought to a simpler normal form (see the proofs of Theorems \ref{thm1} and \ref{thm2} in Section \ref{sec2}). Without loss of generality, in the normal form coordinates we assume the transversals $\{x=1\}$ (vertical) and $\{y=1\}$ (horizontal), and compute the dimension of the family of curves passing through a given sequence of points on the entry transversal and ending on exit transversal.
\medskip

Note that planar saddles and saddle-nodes, unlike planar foci or complex saddles, are not monodromic points, so that the first return map around the saddle/saddle-node point is not well-defined before we close the connection (as in the polycyle). Therefore, our first results in Theorems~\ref{thm1} and \ref{thm2}  concern the box dimension of a union of disjoint local trajectories in the neighborhood of the saddle/saddle-node singularity that correspond to (any) prescribed sequence of points on the entry transversal.
The expression for the Minkowski dimension of spiral trajectories accumulating on a polycycle is provided in Corollary \ref{cor1}.
    \bigskip
    
    Let $s$ be a hyperbolic saddle or a semi-hyperbolic singularity of an analytic vector field.
    By $t_S$ and $t_U$ we denote the transversals to the stable and the unstable manifold of $s$ in the saddle case, i.e. the transversals to the stable (up to the time reversal preserving the geometry of the flow) and the center manifold in the semi-hyperbolic case (in the saddle region). We call $t_S$ also the \emph{entry} and $t_U$ the \emph{exit} transversal, due to obvious reasons. Up to the reversal of the time, without loss of generality we assume that the ratio of hyperbolicity of the saddle is greater than $1$ in the hyperbolic saddle case.
    Up to the change of the axes, we additionally assume that the stable, i.e. entry, transversal is the vertical transversal $\{x=1\}$.
 
    %Let $(y_n)_n$ be a sequence of points on $t_S$. Let $(x_n)_n$ be a sequence in which the spiral trajectories around the saddle/saddle-node $s$ intersect $t_U$.$$ and let $$If $s$ is a vertex of a polycycle, then the spiral trajectories near that polycycle intersect both $t_S$ and $t_U$ in sequences $(y_n)_n$ and $(x_n)_n$ respectively.
    We take $(y_n)_n$ to be any sequence on $t_S$ that converges monotonically to the intersection of $t_S$ with the stable manifold, and such that the distances between consecutive points $y_n$ decrease monotonically (see Figure \ref{fig:saddle}). Let $(x_n)_n$ be the sequence of points where the trajectories $(\Gamma_n)_n$ of the vector field of the saddle/saddle-node $s$ going through $(y_n)_n$ intersect $t_U$.  Under the above assumptions it is not hard to see that $\dimb\, (y_n)_n \geq \dimb\, (x_n)_n$.
    We show that the Minkowski dimension of the union of trajectories $(\Gamma_n)_n$ between the points $y_n$ and $x_n$ is $$\dimb(\cup_n\Gamma_n)=1 + \dimb\, (y_n)_n.$$
    
    \begin{theorem}[Minkowski dimension of the hyperbolic saddle]\label{thm1}
        Let $s=0$ be a hyperbolic saddle of an analytic vector field with ratio of hyperbolicity $\frac{1}{\alpha} \geq 1$:
        $$
        \begin{cases}
        x'=-x + h.o.t.,\\
        y'=\alpha y + h.o.t.
        \end{cases}
        $$ Let $t_S,\ t_U$ and $(y_n)_{n\in\mathbb{N}}$ on $t_S$ be defined as above.
        %Additionally, we assume that the distances between consecutive members of $(y_n)_n$ monotonically decrease. 
        If the sequence $(y_n)_{n\in\mathbb{N}}$ has  Minkowski dimension, $\dimb\, (y_n)_{n}$, then
        \[ \dimb\, \left(\cup_{n\in\mathbb{N}} \Gamma_n\right) = 1 + \dimb (y_n)_n.\]
    \end{theorem}
    
    \begin{theorem}[Minkowski dimension of the semi-hyperbolic singularity]\label{thm2}
        Let $s=0$ be a semi-hyperbolic singularity  of an analytic vector field:
        $$
        \begin{cases}
        x'=-x + h.o.t.,\\
        y'=\alpha y^m + h.o.t.,\ \alpha>0,\ m\geq 2.
        \end{cases}
        $$
        Let $t_S,\ t_U$ and $(y_n)_{n\in\mathbb{N}}$ on $t_S$ be defined as above.
        %Additionally, we assume that the distances between consecutive members of $(y_n)_n$ monotonically decrease.
        If the sequence $(y_n)_{n\in\mathbb{N}}$ has the Minkowski dimension, $\dimb\, (y_n)_{n}$, then
        \[ \dimb\, \left(\cup_{n\in\mathbb{N}} \Gamma_n\right) = 1 + \dimb (y_n)_n.\]
    \end{theorem}
    
    In Propositions~\ref{lm1} and \ref{lm2},  to fix the ideas, we first prove the weaker versions of Theorems~\ref{thm1} and \ref{thm2} for linear saddles:
    \begin{equation}\label{nf1}
        \begin{cases}
            \dot{x} = -x\\
            \dot{y} = \alpha y,\ \ 0 < \alpha \leq 1, 
        \end{cases}, 
    \end{equation}
    and the simplest semi-hyperbolic singularities
    \begin{equation}\label{nf2} \begin{cases}
    \dot{x} = -x\\
    \dot{y} = \alpha y^m,\ \ m\geq 2,\ \alpha > 0.
    \end{cases}\end{equation}

    \begin{proposition}\label{lm1}
        For a linear saddle \eqref{nf1}, 
        under notation of Theorem~\ref{thm1}, it holds that:
        \[ \dimb \left(\cup_{n\in\mathbb{N}} \Gamma_n\right) = 1 + \dimb (y_n)_n. \]
       \end{proposition} 
        \begin{proof}
        
            Since bi-Lipschitz transformations do not change the Minkowski dimensions, using a rescaling in $x$ and $y$ we may assume that $t_S = \{ x = 1 \}$ and $t_U = \{ y = 1\}$.
            The rescaled sequence on $t_S = \{ x = 1 \}$ satisfies the same assumptions as the original one, and we use the same notation $(y_n)_{n\in\mathbb{N}}$.
            We first present the standard computation of the Minkowski dimension of the sequence $(y_n)_{n\in\mathbb{N}}$ in one-dimensional ambient space $\mathbb{R}$.
            Let us denote $Y := \{ y_n \colon n\in\mathbb{N}\}$.
            For $\delta > 0$ small enough, there is a unique critical index $n_\delta$ such that $y_{n_\delta} - y_{n_\delta + 1} < 2\delta$ and $y_{n} - y_{n+1} \geq 2\delta$, for all $n < n_\delta$.
            We now divide the $\delta$-neighborhood $Y_\delta$ into two parts.
            The $\delta$-neighborhoods of points $y_1, y_2, ..., y_{n_\delta - 1}$ do not intersect, and we call their union the tail of $Y_\delta$ and denote it by $T_\delta$.
            On the other hand, the $\delta$-neighborhood of the remainder of the sequence is the interval $(-\delta, y_{n_\delta} + \delta)$, and it is refered to as the nucleus of $Y_\delta$ and denoted by $N_\delta$ (see e.g.\cite{tricot95}).
           Note that $N_\delta$ and $T_\delta$ are disjoint.
            The Lebesgue measure of $Y_\delta$ is now equal to:
            \[ |Y_\delta| = |N_\delta| + |T_\delta| = (y_{n_\delta} + 2\delta) + (n_\delta - 1)\cdot 2\delta = y_{n_\delta} + 2\delta n_\delta. \]
            Since, by our assumptions, $Y$ has  Minkowski dimension, by definition of Minkowski dimension it holds that
            \[ \dimb Y = \lim_{\delta \to 0} \left[ 1 - \frac{ \ln (y_{n_\delta} + 2\delta n_\delta)}{\ln \delta} \right]. \]
            
            Let $$\Gamma := \{\Gamma_n \colon n\in\mathbb{N}\}.
            $$
            
            Let us first consider the part of the set $\Gamma$ of trajectories in the region $\{ \frac{1}{2} \leq x \leq 1\}$.
            Since the Minkowski dimension of the Cartesian product is the sum of Minkowski dimensions, in this region the Flow-Box Theorem allows us to easily calculate the Minkowski dimension to be $1 + \dimb Y$.
            Due to monotonicity of $\ldimb$, we now have the lower bound
            $$1 + \dimb Y \leq \ldimb \Gamma.$$
            
            Therefore, to prove the proposition, it suffices to show that \begin{equation}\label{eq:d}\overline{\dim}_B \Gamma \leq 1 + \dimb Y.\end{equation} It can be shown that there exists a positive constant $C > 0$ such that, for any $n\in\mathbb{N}$, the Lebesgue measure of $(\Gamma_n)_\delta$ is bounded from above by $C\delta$, as $\delta \to 0$.
            This bound is uniform both in $\delta\in(0,\delta_0)$, $\delta_0>0$, and in $n\in\mathbb N$.
            Therefore, for a given $\delta > 0$, we bound the Lebesgue measure of the $\delta$-neighborhood of the union of trajectories $\cup_{ n<n_\delta}\Gamma_n$, arising from the tail of $Y_\delta$, from above by 
            \begin{equation}\label{eq:tl}\left|\cup_{n<n_\delta}(\Gamma_n)_\delta\right|\leq C\delta(n_\delta-1).\end{equation} On the other hand, the Lebesgue measure of the $\delta$-neighborhood of the union of trajectories $\cup_{n\geq n_\delta}\Gamma_n$, arising from the nucleus of $Y_\delta$, is bounded from above by 
            \begin{equation}\label{nucleusbound} \left|\cup_{n\geq n_\delta}(\Gamma_n)_\delta\right|\leq y_{n_\delta} + \int_{y_{n_\delta}}^1 x(y)dy + D\delta, \end{equation}
            where $D$ is a universal positive constant and $y \mapsto x(y)$ is a function whose graph is the curve $\Gamma_{n_\delta}$.
            For more details, see Figure \ref{fig:saddle}.

            \begin{figure}[h!]
                \centering
                \includegraphics[scale = 0.1]{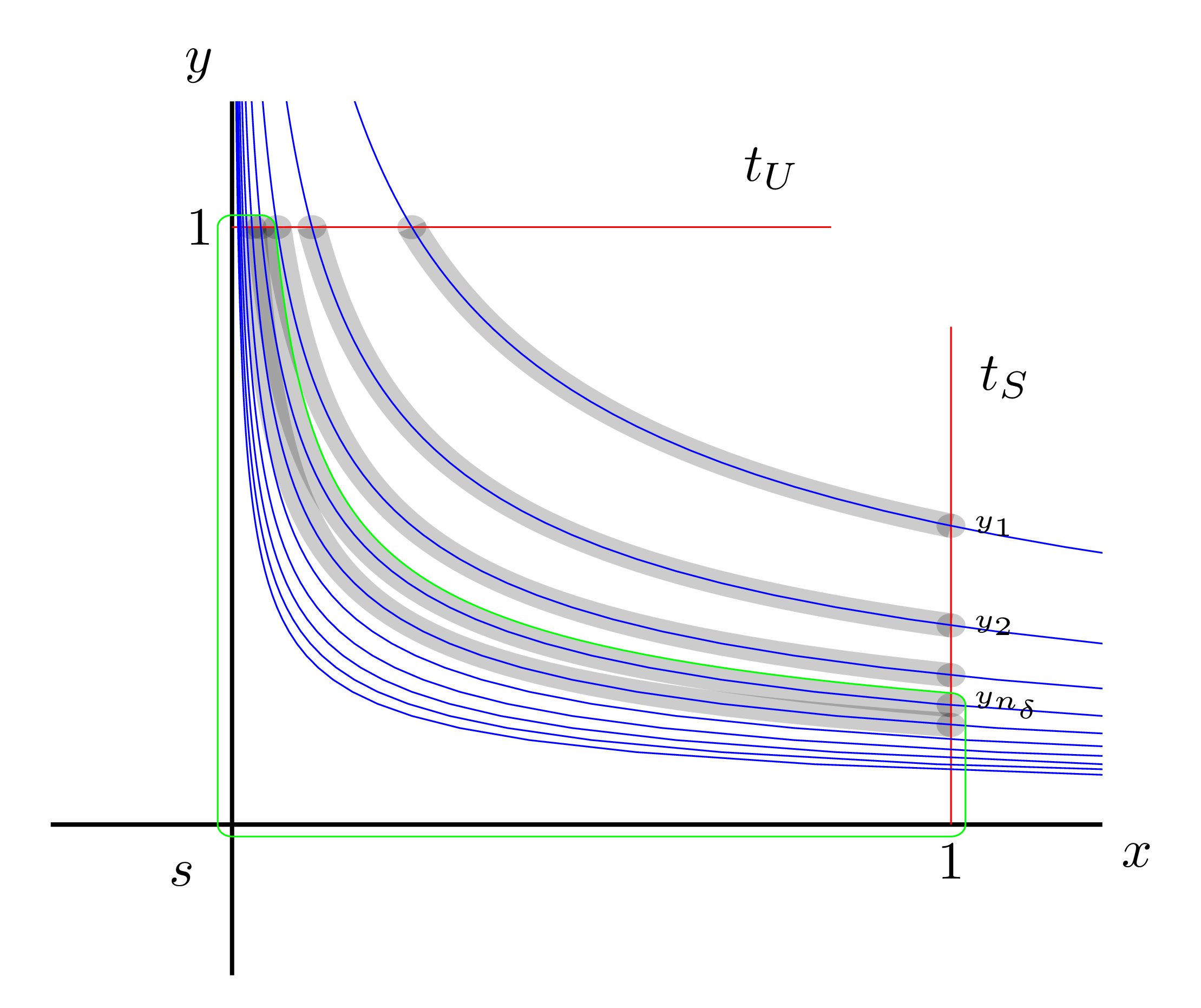}
                \caption{$\delta$-neighborhood of $\Gamma$}
                \label{fig:saddle}
            \end{figure}
            
\smallskip
            
            \noindent To prove inequality \eqref{eq:d}, we consider separately two cases.
            \begin{enumerate}
            \item{Case: $0 < \alpha < 1$.}\\
            By direct integration of the simple vector field (\ref{nf1}), we get
            \[ x(y) = \left( \frac{y_{n_\delta}}{y}\right)^{\frac{1}{\alpha}}, \]
            so 
            \begin{equation}{\label{nucleusintegral}}
                \int_{y_{n_\delta}}^1 x(y)dy = \frac{\alpha}{1-\alpha}( y_{n_\delta} - y_{n_\delta}^{\frac{1}{\alpha}}).
            \end{equation}
            Therefore, using \eqref{eq:tl}, \eqref{nucleusbound} and \eqref{nucleusintegral}, there exists a universal positive constant $M > 0$ such that
            \[ |\Gamma_\delta| \leq M(y_{n_\delta} + 2\delta n_\delta),\ \delta > 0. \]
            Finally,
           \begin{align*} \limsup_{\delta \to 0} \left[ 2 - \frac{\ln |\Gamma_\delta|}{\ln \delta} \right] &\leq \limsup_{\delta \to 0} \left[ 2 - \frac{ \ln M + \ln(y_{n_\delta} + 2\delta n_\delta)}{\ln \delta} \right]\\
            &= \limsup_{\delta \to 0} \left[ 2 - \frac{\ln(y_{n_\delta} + 2\delta n_\delta)}{\ln \delta} \right] = 1 + \dimb Y. \end{align*}
            Therefore, $\overline{\dim}_B \Gamma \leq 1 + \dimb Y$. %This completes the proof of Lemma \ref{lm1} in the case $\alpha < 1$.

            \item{Case: $\alpha=1$.}\\
            Note that formula (\ref{nucleusintegral}) cannot be used. However, similarly as in Case 1, by direct integration we get:
            $$
            x(y)=\frac{y_{n_\delta}}{y},\ \ \int_{y_{n_\delta}}^1 x(y)dy =y_{n_\delta}(-\ln y_{n_\delta}).
            $$
            Therefore, there exists $M>0$ such that:
            $|\Gamma_\delta|\leq M(y_{n_\delta}(-\ln y_{n_\delta})+2\delta n_\delta)$.
            Now, since both  $y_{n_\delta}\to 0$ and $\delta n_\delta\to 0$, as $\delta \to 0$, for every small $\kappa>0$ there exists $\delta_\kappa$, such that, for every $0<\delta<\delta_\kappa$, it holds that 
            $$
            y_{n_\delta}(-\ln y_{n_\delta})<y_{n_\delta}^{1-\kappa},\ 2\delta n_\delta<(2\delta n_\delta)^{1-\kappa},\ \delta< \delta_\kappa.
            $$
            Therefore,
            $$
            |\Gamma_\delta|\leq 2M \max\{y_{n_\delta},2\delta n_\delta\}^{1-\kappa},\ \delta<\delta_{\kappa}.
            $$
            Now, for every $\kappa>0$,
            \begin{align*} \limsup_{\delta \to 0} &\left[ 2 - \frac{\ln |\Gamma_\delta|}{\ln \delta} \right] \leq \limsup_{\delta \to 0} \left[ 2 - \frac{ \ln (2M)+(1-\kappa)\ln\max\{y_{n_\delta},2\delta n_\delta\}}{\ln \delta}\right]\\
            &\leq \limsup_{\delta \to 0} \left[ 2 - (1-\kappa)\frac{\ln(y_{n_\delta} + 2\delta n_\delta)}{\ln \delta} \right] = 1 +\kappa+(1-\kappa) \dimb Y. \end{align*}
            Letting $\kappa\to 0$, we get
            $$\limsup_{\delta \to 0} \left[ 2 - \frac{\ln |\Gamma_\delta|}{\ln \delta} \right]\leq 1+\dim_B Y.$$
            \end{enumerate}
             Finally, since
        \[ 1 + \dimb Y \leq \ldimb \Gamma \leq \udimb \Gamma \leq 1 + \dimb Y, \]
        we conclude that
        \[ \dimb \Gamma = 1 + \dimb Y. \]
\end{proof}

    \begin{proposition}\label{lm2}
        For a semi-hyperbolic singularity \eqref{nf2}, 
        under notation of Theorem~\ref{thm2}, it holds that:
        \[ \dimb \left(\cup_{n\in\mathbb{N}} \Gamma_n\right) = 1 + \dimb (y_n)_n. \]
\end{proposition}
    \begin{proof}
        We use a similar nucleus-tail approach from the proof of Lemma \ref{lm1}.
        Again, rescaling $x$ and $y$, we assume $t_S = \{ x = 1 \}$ and $t_U = \{ y = 1 \}$.
        This changes the constant $\alpha$ in \eqref{nf2}, but $\alpha$ remains positive.
        Again we denote $Y := \{ y_n \colon n\in\mathbb{N} \}$ and $\Gamma = \{ \Gamma_n \colon n\in\mathbb{N}\}$.
        As in the proof of Lemma \ref{lm1}, $\underline{\dim}_B \Gamma \geq 1 + \dimb Y$.
        Again, there exist uniform positive constants $C$ and $D$ such that same bounds \eqref{eq:tl} and \eqref{nucleusbound} hold.
        Let us show that the integral in (\ref{nucleusbound}) satisfies
        \begin{equation}\label{smallintegral} \int_{y_{n_\delta}}^1 x(y)dy = o(y_{n_\delta}),\ \delta \to 0, \end{equation}
        where $x(y) = \exp\left( \frac{ y^{1-m} - y_{n_\delta}^{1-m}}{\alpha(m-1)} \right)$ is the function whose graph is the trajectory $\Gamma_{n_\delta}$ through $(1,y_{n_\delta})$.
        To prove this it suffices to show that the function 
        \[ y \mapsto \int_y^1 \exp\left( \frac{t^{1-m} - y^{1-m}}{\alpha(m-1)}\right)dt \]
        is $o(y)$, as $y \to 0$. Indeed, we have that
        \begin{align*} \lim_{y \to 0} \frac{ \int_y^1 \exp\left( \frac{t^{1-m} - y^{1-m}}{\alpha(m-1)}\right)dt }{y} &= \lim_{y \to 0} 
            \frac{ \int_y^1 \exp\left( \frac{t^{1-m}}{\alpha(m-1)}\right)dt}{y\exp\left( \frac{y^{1-m}}{\alpha(m-1)}\right)} = \\
             &= \lim_{y \to 0} \frac{ - \exp\left( \frac{y^{1-m}}{\alpha(m-1)}\right)}{(1-\frac{y^{1-m}}{\alpha})\exp\left( \frac{y^{1-m}}{\alpha(m-1)}\right)} = 0, \end{align*}
        where we used the L'Hospital's rule in the second step.
        Using (\ref{smallintegral}) as in Lemma \ref{lm1}, we get that $|\Gamma_\delta| \leq M(y_{n_\delta} + 2\delta n_\delta)$ for some positive constant $M$. Consequently, $\udimb \Gamma \leq 1 + \dimb Y$.
        \end{proof}

    \begin{proof}[Proof of Theorem \ref{thm1}]
       Consider an analytic vector field with a hyperbolic saddle, and let $t_S$ and $t_U$ be as in the statement of the theorem.
        Following \cite[Theorem 2.15]{dumllibart06}, near the hyperbolic saddle the field can be reduced to the following (smooth) orbital normal form:
        \begin{equation}\label{eq:nf} \begin{cases}
        \dot{x} = -x\\
        \dot{y} = \alpha y + h(x,y), \end{cases}
        \end{equation}
        where $\frac{1}{\alpha} \geq 1$ is the hyperbolicity ratio of the saddle and $h(x,y) = O(xy^2),\ (x,y)\to (0,0),$ is a $C^\infty$ function. Since a smooth local change of coordinates is bi-Lipschitz, it preserves the Minkowski dimension around $0$. Therefore, it suffices to compute Minkowski dimension of the saddle in the normal form \eqref{eq:nf}.
        
        For an arbitrary $\beta > 0$, we choose a small enough neighborhood of the saddle such that $|\alpha y + h(x,y)| \leq (1+\beta)y$.
        Now we choose new transversals $t_S'$ and $t_U'$ that intersect the stable and unstable manifold respectively in this neighborhood.
        Up to a rescaling, we may assume that $t'_S = \{ x = 1\}$ and $t'_U = \{ y = 1\}$.
        Due to the Flow-Box Theorem, the maximal Minkowski dimension of parts of $\Gamma = \cup_{n\in\mathbb{N}} \Gamma_n$ between $t_S$ and $t_S'$,  $t_U$ and $t_U'$, as well as of those around $t_U'$ and $t_S'$, is equal to $1 + \dimb Y$. Therefore,
        \begin{equation}\label{eq:a} \ldimb \Gamma\geq 1 + \dimb Y. \end{equation}
        To prove the theorem, it suffices to show that, for every $\beta>0$, \begin{equation}\label{eq:h}\overline{\dim}_B \Gamma' \leq \frac{2\beta + 1 + \dimb Y}{1+\beta},\end{equation} where $\Gamma'$ is the part of $\Gamma$ between $t_S'$ and $t_U'$.
        Evidently, $\overline{\dim}_B \Gamma'=\overline\dim_B\Gamma$. Passing to limit as $\beta\to 0$, we get:
        $$
        \overline{\dim}_B \Gamma\leq 1+\dim_B Y,
        $$
        which, along with \eqref{eq:a}, concludes the proof of the theorem.
        \medskip
        
    Let us now prove \eqref{eq:h}. For $(u,v)\in\Gamma_n$ we have that
        \[ - \ln u = \int_1^u -\frac{dx}{x} = \int_{y_n}^{v} \frac{dy}{\alpha y + h(x,y)} \geq \int_{y_n}^{v} \frac{dy}{(1+\beta)y}, \]
        that is, 
        \[ u \leq \left( \frac{y_n}{v} \right)^\frac{1}{1+\beta}. \]
       Similarly as in the proof of Proposition~\ref{lm1}, the Lebesgue measure of the $\delta$-neighborhood of trajectories $\Gamma_n$ arising from the nucleus is bounded above by
        \[ |\cup_{n\geq n_\delta}(\Gamma_n)_\delta|\leq y_{n_\delta} + \int_{y_{n_\delta}}^1 \left( \frac{y_{n_\delta}}{y}\right)^\frac{1}{1+\beta}dy + D\delta. \]
        The integral above is of order $O(y_{n_\delta}^\frac{1}{1+\beta})$, as $\delta \to 0$. Now we proceed similarly as in the proof of \emph{Case 2.} in Proposition~\ref{lm1}. There exists $M>0$ such that, for sufficiently small $\delta>0$,
        $$
        |\Gamma_\delta|\leq M (y_{n_\delta}^{\frac{1}{1+\beta}}+2\delta n_\delta).
        $$
        Due to the fact that, for every $\beta>0$, $2\delta n_\delta<(2\delta n_\delta)^{\frac{1}{1+\beta}}$ for sufficiently small $\delta<\delta_\beta$, we get 
        $$
        |\Gamma_\delta|\leq 2M\max\{y_{n_\delta},2\delta n_\delta\}^{\frac{1}{1+\beta}},\ \delta<\delta_\beta.
        $$
        Using exactly the same procedure as in the proof of Proposition~\ref{lm1}, \emph{Case 2.}, we finally get
        $$
\limsup_{\delta \to 0} \left[ 2 - \frac{\ln |\Gamma_\delta|}{\ln \delta} \right] \leq 2-\frac{1}{1+\beta}(1-\dim_B Y)=\frac{2\beta+1+\dim_B Y}{1+\beta}.
        $$
         \end{proof}

    \begin{proof}[Proof of Theorem \ref{thm2}]
        Due to \cite[Theorem 2.19]{dumllibart06} we can use the following orbital normal form near the semi-hyperbolic singularity
        \[ \begin{cases}
        \dot{x} = -x,\\
        \dot{y} = y^m + h(x,y),\ m\geq 2, \end{cases}\]
        where $h(x,y) = O(y^{2m-1})$ is a $C^\infty$ function. %{\color{red} Note that, by a linear change in $y$, we may assume $\alpha=1$. }
       
        Now, similarly as in the proof of Theorem \ref{thm1}, we find a sufficiently small neighborhood of the singularity where $| h(x,y) | \leq y^m$.
        Now, with this bound, we proceed similarly as in the proof of Proposition~\ref{lm2} to show that $\udimb \Gamma \leq 1 + \dimb Y$.
        The other inequality, $\underline{\dim}_B \Gamma\geq 1+\dim_B Y$, is deduced analogously as in Theorem~\ref{thm1}. %Since parts of $(\Gamma_n)_n$ between $t_S$ and $t_S'$, and $t_U$ and $t_U'$ have Minkowski dimension at most $1 + \dimb (y_n)_n$ this immediately yields the desired result.
    \end{proof}

    \begin{corollary}[Minkowski dimension of a monodromic polycycle]\label{cor1}
        Let $P$ be a monodromic $N$-polycycle of an analytic vector field, with hyperbolic saddles and semi-hyperbolic singularities as vertices, $N\in\mathbb N$. %We assume that spiral trajectories are accumulating on the polycycle. 
        Let $S$ be a spiral trajectory accumulating to $P$. Let $t_1, t_2, ...., t_N$ be transversals to $($all$)$ regular sides of the polycycle. By $Y_k,\, k\in\{1,2,...,N\}$, we denote the intersections of $S$ with transversals $t_k$ respectively.
        Then,
        \[ \dimb\, S = 1 + \max \left\{ \dimb\, Y_k \colon k\in\{1,2,...,N\} \right\}. \]

    \begin{proof}
        This Corollary is a direct consequence of Theorems \ref{thm1} and \ref{thm2} and the finite stability of the Minkowski dimension.
    \end{proof}
    \end{corollary}

\section{Applications to hyperbolic saddle-loops}\label{sec3}
    The simplest polycycle configuration is a hyperbolic saddle-loop of an analytic vector field.
    A saddle-loop is an invariant set in which the unstable manifold of the saddle extends to its stable manifold (i.e. there exists a homoclinic connection of the saddle).
    
    Let $r>0$ be the hyperbolicity ratio of the saddle. It is well-known (see e.g. \cite[pg. 109]{roussarie98}) that the first return map $P$ on any transversal to the regular part of the loop (parametrized regularly by $x\in [0, T[$, where $x = 0$ corresponds to the point on the loop) satisfies:
    \begin{enumerate}
    \item (codimension $1$ case, $r \in \mathbb{R}^+\setminus\{ 1 \}$)
    \[ P(x) \, \sim Ax^r,\ A>0, \]
    \item (higher finite codimension cases, $r=1$)
    \[ P(x) = x + \delta(x),\]
    where
    \[ \delta(x) = \beta_1 x + \alpha_2x^2(-\ln x) + \beta_2x^2 + ... + \beta_{k-1}x^{k-1} + \alpha_kx^k(-\ln x) + O(x^k). \]
\end{enumerate}
The saddle loop is said \cite{roussarie98} to be of \emph{codimension $2k$} if $\delta(x) \sim \beta_kx^k$, as $x\to 0$, $\beta_k \neq 0$. It is said to be of \emph{codimension $2k+1$} if $\delta(x) \sim \alpha_{k+1}x^{k+1}(-\ln x)$, as $x\to 0$, with $\alpha_{k+1}\neq 0$, $k\geq 1$. We exclude here \emph{infinite codimension} cases when $P\equiv \mathrm{id}$, that is, when the loop is of \emph{center type}. In other words, in \emph{finite codimension} cases there is an accumulating spiral trajectory to the loop. In \cite{roussarie98} it is shown that the codimension of the saddle loop corresponds to its cyclicity in generic unfoldings.
\medskip

In the following Theorem~\ref{thm3}, we apply our results from Section~\ref{sec2} to give a correspondence between the codimension of the loop and the Minkowski content of its (any) spiral trajectory. The correspondence between the codimension of the saddle-loop and the Minkowski dimension of spiral trajectories is $2$--$1$.
    \begin{theorem}[Minkowski dimension of a saddle-loop]\label{thm3}
        The Minkowski dimension of a spiral trajectory $S$ in an analytic vector field that has a finite-codimension saddle-loop as its $\alpha/\omega$-limit set depends only on the codimension of the saddle-loop.
        More precisely, if $k\geq 1$ is the codimension of the saddle loop, then:
        \[ \dimb S = \begin{cases}
        2 - \frac{2}{k},& k \text{ even,}\\
        2 - \frac{2}{k+1},& k \text{ odd}. \end{cases} \]
\end{theorem}
    \begin{proof}
        In codimension $1$ and $2$ cases,  the first return map $P$ is hyperbolic, so the Minkowski dimension of the set of points of intersection of any spiral with a regular transversal to the saddle-loop (i.e. of an orbit of $P$) is $0$ (see \cite[Lemma 1]{neveda07}). Now, due to the Flow-Box Theorem, Theorem~\ref{thm1} and the finite stability of the Minkowski dimension, the dimension of the spiral is \emph{trivial}, $\dim_B S=1$.

        For even codimensions $k>2$, the Minkowski dimension of the set of points of intersection of any spiral with a transversal to the saddle-loop (i.e. of an orbit of $P$) is $1-\frac{2}{k}$ due to \cite[Theorem 1]{neveda07}. Again, using the Flow-Box Theorem and Theorem \ref{thm1} we conclude that $\dim_B S=2 - \frac{2}{k}$.
        
        For odd codimensions $k>2$, the Minkowski dimension of the set of points of intersection of any spiral with a transversal to the saddle-loop is $1-\frac{2}{k+1}$ (see \cite[Theorem 2]{MarResZup12}). Therefore, $\dim_B S=2 - \frac{2}{k+1}.$
    \end{proof}

\section{Applications to hyperbolic $2$-cycles}\label{sec4}
    In this section we focus on analytic vector fields with a hyperbolic $2$-saddle polycycle $\Gamma_2$ with hyperbolicity ratios $r_1$ and $r_2$ (see Figure \ref{fig:2cycle}).
    We apply our fractal methods to the classical results from \cite{mourtada94} (and references therein) about the cyclicity of the $2$-cycle in the \emph{non-degenerate} and the \emph{degenerate} case.

    \begin{figure}[h!]
        \centering
        \includegraphics[scale = 0.6]{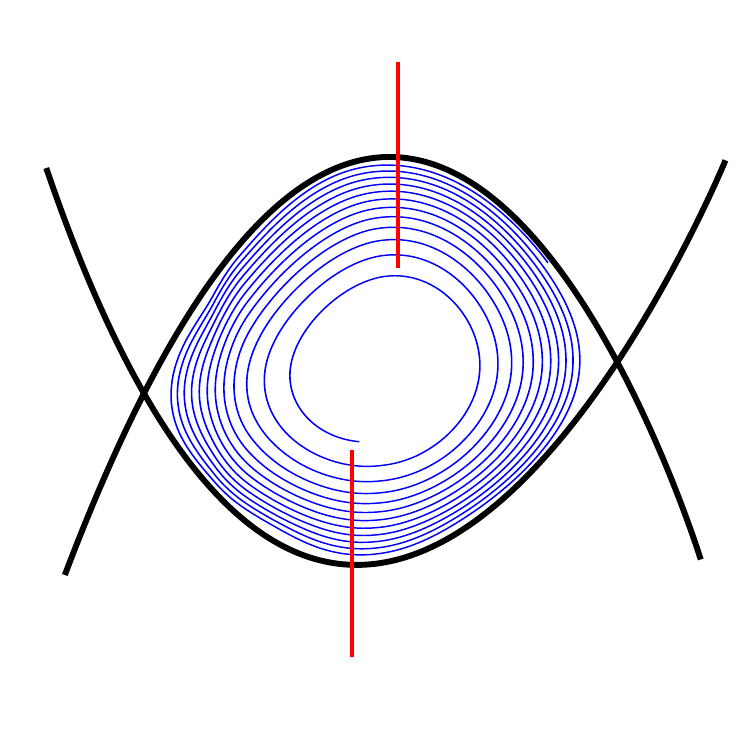}
        \caption{A non-trivial hyperbolic $2$-cycle}
        \label{fig:2cycle}
    \end{figure}

\subsection{\emph{Non-degenerate} $2$-cycles}
    \begin{theorem}[Cyclicity of non-degenerate $2$-cycles, \cite{mourtada94}]\label{teo0}
        If the conditions
        \begin{equation}\label{nondeg}
            r_1 \neq 1, \quad r_2 \neq 1, \quad r_1r_2 \neq 1
        \end{equation}
        hold, then the polycycle $\Gamma_2$ is of cyclicity less than or equal to $2$ in any $C^\infty$-unfolding. If, moreover, 
        $$(r_1-1)(r_2-1) < 0,$$
        there exists a two-parameter $C^\infty$-versal unfolding $(X_\lambda)$ in which $\Gamma_2$ is of cyclicity $2$. Otherwise, there exists a two-parameter $C^\infty$-versal unfolding in which $\Gamma_2$ is of cyclicity $1$.
    \end{theorem}

    We now state our fractal result for non-degenerate $2$-cycles from Theorem \ref{teo0}.
    Note that, since $r_1r_2\neq 1$, the $2$-cycle of Theorem~\ref{thm5} is not of \emph{center-type} ($P\neq\mathrm{id}$), but has accumulating spiral trajectories (\emph{focus-type}).

    \begin{theorem}[Minkowski dimension of a non-degenerate $2$-polycycle]\label{thm5}
       Let $\Gamma_2$ be a monodromic $2$-polycycle of an analytic vector field, non-degenerate in the sense of \eqref{nondeg}. The Minkowski dimension of any spiral trajectory accumulating on $\Gamma_2$ is \emph{trivial}:
       $$\dim_B S=1.$$
    \end{theorem}
        \begin{proof}
            Let $t_1$ and $t_2$ be any regularly parametrized transversals to heteroclinic connections of $\Gamma_2$ not intersecting the saddles.
            The first return maps $P_i:t_i \to t_i,\ i\in\{1,2\}$, as compositions of regular diffeomorphisms and corner maps of the saddles, satisfy
            \[ P_i(x) \simeq x^{r_1r_2},\, x \to 0. \]
            Due to \cite[Lemma 1]{neveda07}, the Minkowski dimension of its orbit (i.e. of the intersection of the spiral with $t_1$ and $t_2$) is $0$.
            By Corollary \ref{cor1}, the Minkowski dimension of the entire spiral is $1$.
        \end{proof}

\subsection{\emph{Degenerate} $2$-cycles}
    In \cite{mourtada94}, Mourtada distinguishes between three families of \emph{degenerate} $2$-cycles (when some of non-degeneracy conditions $r_1\neq 1,\, r_2\neq 1$ or $r_1r_2 \neq 1$ do not hold):
\begin{itemize}
    \item $\mathcal{C}_1 = \{ \Gamma_2 \colon r_1r_2 \neq 1 \}$,
    \item $\mathcal{C}_2 = \{ \Gamma_2 \colon r_1r_2 = 1,\, r_1 \not\in \mathbb{Q} \}$,
    \item $\mathcal{C}_3 = \{ \Gamma_2 \colon r_1r_2 = 1,\, r_1 \in \mathbb{Q} \}$.
\end{itemize}
In the remainder of this paper we focus only on fractal analysis on families $\mathcal{C}_1$ and $\mathcal{C}_2$ and comparison with their known cyclicities. The cyclicity of polycycles in $\mathcal{C}_3$ is more complicated due to the presence of independent \emph{Ecalle-Roussarie compensators} from two resonant saddles (see \cite{roussarie98}).
    
\subsubsection{Family $\mathcal C_1$}
    \begin{theorem}[Cyclicity of $\mathcal C_1$, Theorem 1 in \cite{mourtada94}]\label{teo1}
        Let $\Gamma_2$ be a $2$-cycle belonging to $\mathcal{C}_1$ and tangent to a planar vector field $X_0$.
        Then $\Gamma_2$ is of cyclicity less than or equal to $2$ in every $C^\infty$ family $(X_\lambda)$ unfolding $X_0$.
        Furthermore, there exists a three-parameter $C^\infty$-versal unfolding $(X_\lambda)$ in which $\Gamma_2$ is of cyclicity $2$.
    \end{theorem}

    Again, same as in the case of non-degenerate $2$-cycles in Theorem~\ref{thm5}, it is clear here (by form of $P\neq\mathrm{id}$) that $\Gamma_2\in\mathcal C_1$ is not of center-type, but has spiral trajectories.

    \begin{theorem}[Minkowski dimension of $\mathcal C_1$]\label{thm6}
        Let a degenerate $2$-cycle $\Gamma_2$  of an analytic vector field belong to the family $\mathcal{C}_1$.
        The Minkowski dimension of any spiral trajectory $S$ accumulating on $\Gamma_2$ is \emph{trivial}, $\dim_B S=1$.
    \end{theorem}
        \begin{proof}
            The proof is analogous to the proof of Theorem \ref{thm5}. 
        \end{proof}

\subsubsection{Family $\mathcal C_2$}
%\begin{theorem}[Theorem 2 in \cite{mourtada94}]\todo{ZAPRAVO NE RAZUMIJEM ZASTO OVAJ PRVI TM, TREBA LI? JA BIH MAKNULA.}
 %       Let $\Gamma_2$ be  a non-trivial $($in the sense that the first return map is not equal to the identity$)$ $2$-cycle in $\mathcal{C}_2$ tangent to an analytical planar vector field $X_0$. There exists $N\in\mathbb{N}$, depending only on the germ of $X_0$ along $\Gamma_2$, such that the cyclicity of $\Gamma_2$ is less than or equal to $N$ in every $C^\infty$ family $(X_\lambda)$ unfolding $X_0$.
  %  \end{theorem}

First note that here, since $r_1r_2=1$, an extra assumption of \emph{non-trivial} polycycles is requested in Theorems \ref{mrt3} and \ref{thm4} to exclude identity first return maps, that is, $2$-cycles that are of center type. Therefore, we consider only monodromic cases of polycycles with spiraling trajectories.
\medskip

    Let $\Gamma_2\in\mathcal C_2$ be a $2$-cycle of an analytic vector field. The following analysis of the first return maps on  transversals to both heteroclinic connections is due to Mourtada \cite{mourtada94}. Since the ratios of hyperbolicity $r_1$ and $r_2$ of the saddles are irrational, the saddles are analytically linearizable, and there are $C^\infty$ transversals $\sigma_i, \tau_i$ near the saddles such that the associated corner Dulac maps $D_i : \sigma_i \to \tau_i$ are given by
    \[ y_i = D_i(x_i) = x_i^{r_i},\, i = 1,2. \]
    For more details see \cite[pg. 78]{mourtada94}.
    On the other hand, regular transition maps $R_1: \tau_2\to \sigma_1$ and $R_2 : \tau_1 \to \sigma_2$ are given by
    \[ x_1 = R_1(y_2) = \beta_{2,1}y_2\left[ 1 + \alpha_1 y_2^{k_1-1} + o(y_2^{k_1-1})\right]\]
    and
    \[ x_2 = R_2(y_1) = \beta_{1,2}y_1\left[ 1 + \alpha_2 y_1^{k_2-1} + o(y_1^{k_2-1}) \right] \]
    where $\beta_{1,2}, \beta_{2,1} > 0$ and $2 \leq k_i \in \mathbb{N}\cup \{ \infty \}$, and where $k_i = \infty$ implies $\alpha_i = 0$.
    The first return maps $P_1 = R_1 \circ D_2 \circ R_2 \circ D_1 : \sigma_1 \to \sigma_1$ and $P_2 = R_2 \circ D_1 \circ R_1 \circ D_2 : \sigma_2 \to \sigma_2$ are then given by:
   \begin{align}\label{p1} P_1(x_1) = \beta_{2,1}\beta_{1,2}^{r_2}x_1\Big[ 1 +&\left(\alpha_1\beta_{1,2}^{r_2(k_1-1)}x_1^{k_1-1} +o(x_1^{k_1-1})\right)+\nonumber\\ 
   &\qquad\qquad +\left(r_2\alpha_2x_1^{r_1(k_2-1)} + o(x_1^{r_1(k_2-1)})\right) \Big], 
   \end{align}
    and
   \begin{align}\label{p2}P_2(x_2) = \beta_{1,2}\beta_{2,1}^{r_1}x_2\Big[ 1 + &\left(\alpha_2\beta_{2,1}^{r_1(k_2-1)}x_2^{k_2-1} +  o(x_2^{k_2-1})\right)+\nonumber\\
   &\qquad\qquad +\left(r_1\alpha_1x_2^{r_2(k_1-1)}  + o(x_2^{r_2(k_1-1)}) \right)\Big].
   \end{align}

    Notice that $P_1$ is hyperbolic/tangent to identity if and only if $P_2$ is hyperbolic/tangent to identity.
    Indeed,
    \[ \beta_{1,2}\beta_{2,1}^{r_1} = \beta_{1,2}^{r_1r_2}\beta_{2,1}^{r_1} = \left( \beta_{2,1}\beta_{1,2}^{r_2}\right)^{r_1} .\]
    Moreover, exactly one of the inequalities $k_1-1 < r_1(k_2-1)$ or $k_2-1 < r_2(k_1-1)$ holds.
    On the contrary, if we assume that both hold, we get
    \[ k_1 - 1 < r_1(k_2-1) < r_1r_2(k_1-1) = k_1-1, \]
    which is obviously a contradiction. On the other hand, if we assume that neither one holds, we have
    \[ k_1-1 \geq r_1(k_2-1) \geq r_1r_2(k_1-1) = k_1-1. \]
    which would imply $k_1-1 = r_1(k_2-1)$ and $k_2-1 = r_2(k_1-1)$. As a consequence, $r_1,\, r_2\in\mathbb Q$, which is a contradiction.
\smallskip

     Moreover, in the case $\beta_{1,2}\beta_{2,1}^{r_1}=1$, $|\alpha_1| +  |\alpha_2| \neq 0$. Otherwise $P_1=P_2=\mathrm{id}$, which is the trivial case that is not considered here. This is a consequence of the \emph{quasi-analyticity} of the first return maps around hyperbolic saddle polycycles in analytic planar vector fields \cite{ilya91}, that states that the Taylor map that associates to a Dulac germ its Dulac asymptotic expansion is injective (i.e., trivial expansion implies the trivial germ).

    \bigskip
    
% Theorem~\ref{mrt3} below is a result on cyclicity of $\Gamma_2\in\mathcal C_2$ by Mourtada. 
    \begin{theorem}[Cyclicity in $\mathcal C_2$, \cite{mourtada94}, p.\,83]\label{mrt3}
        Let $X_0$ be an analytic vector field with a non-trivial $($in the sense that the first return map is not equal to the identity$)$ $2$-cycle $\Gamma_2\in\mathcal C_2$ tangent to $X_0$.
        In the notation as above, 
        \begin{enumerate}
        \item If $\beta_{1,2}\beta_{2,1}^{r_1} \neq 1$, then the cyclicity of $\Gamma_2$ $($in any $C^\infty$ unfolding $(X_\lambda))$ is not greater than $3$.

        \item If $\beta_{1,2}\beta_{2,1}^{r_1} = 1$ and $ |\alpha_1| + |\alpha_2| \neq 0$, then the cyclicity of $\Gamma_2$ $($in any $C^\infty$ unfolding $(X_\lambda))$ is not greater than $\epsilon$, where:
        \[  \epsilon := \begin{cases}
        2 + k_1 + \lfloor \frac{k_1-1}{r_1} \rfloor& \text{if } k_1 - 1 < r_1(k_2-1),\\
        2 + k_2 + \lfloor \frac{k_2-1}{r_2} \rfloor& \text{if } k_2 - 1 < r_2(k_1-1). \end{cases} \]
        \end{enumerate}
    \end{theorem}
    Note the surprising fact that the cyclicity, unlike in all the previous cases, cannot be read only from a single first return map.
\medskip

    We now state our 'fractal' version of Theorem~\ref{mrt3}. The goal is to read the Mourtada's upper bound on the cyclicity of the $2$-cycle $\Gamma_2\in\mathcal C_2$  from the Minkowski dimension of only one trajectory accumulating to $\Gamma_2$. However, as can be expected in the light of the comment before Theorem~\ref{mrt3}, the dimension of the spiral trajectory will not suffice. In order to read Mourtada's upper bound, we need additional fractal data, see Corollary~\ref{obs:a} below.

    \begin{theorem}[Fractal version of Theorem~\ref{mrt3}]\label{thm4}
        Let $X_0$ be an analytic vector field with a non-trivial $2$-cycle $\Gamma_2\in\mathcal{C}_2$.
        %Using the notation above, suppose that $\beta_{1,2}\beta_{2,1}^{r_1} \neq 1$ or $|\alpha_1| + |\alpha_2| \neq 0$.
        Let $r:=\min\{r_1,\,r_2\}$ be the minimal hyperbolicity ratio. Any spiral trajectory accumulating to the polycycle has the same Minkowski dimension $d \in [1,2)$.
        Moreover, the cyclicity of $\Gamma_2$ in $C^\infty$ unfoldings of $X_0$ is at most
        \begin{equation}\label{cycl}
        \left\lfloor 3 + (1+r)\frac{d-1}{2-d} \right\rfloor. \end{equation}
          %  \item If $d = 1$, then the cyclicity of $\Gamma_2$ in {\color{red}$C^\infty$}\todo{Ovdje i svuda u tm-u valjda $C^\infty$ unfolding kao i kod Mourtade?} unfoldings of $X_0$ is at most $3$.

            %\item If $d \in \mathbb{Q}\cap[1,2)$\todo{I dimenzija $1$ moze pod ovaj rezultat, ispadne $3$, zasto odvajas? Moze li dimenzija spirale uopce biti $2$, veca od $2$ svakako ne moze? Cini mi es da ne, jer bi onda na transverzali orbita first return mapa imala dim=1, a to nema ni hip ni parab trajektorija! Ja bih napisala $2$ slučaja: $d\in[1,2)\cap \mathbb Q$ ili $d\in(1,2)\cap (\mathbb R\setminus \mathbb Q)$}, then the cyclicity of $\Gamma_2$ in $C^\infty$ unfoldings of $X_0$ is at most
        %\[ 3 + \frac{d-1}{2-d} + \left\lfloor r\frac{d-1}{2-d} \right\rfloor.\]

            %\item If $d \in (\mathbb R\setminus\mathbb{Q})\cap (1,2)$, then the cyclicity of $\Gamma_2$ in $C^\infty$ unfoldings of $X_0$ is at most
        %\[ 3 + r\frac{d-1}{2-d} + \left\lfloor \frac{d-1}{2-d} \right\rfloor. \]
        %\end{enumerate}
           \end{theorem}
           
        \begin{proof}
            By Theorem \ref{mrt3}, if $\beta_{1,2}\beta_{2,1}^{r_1} \neq 1$ then the cyclicity of the polycycle is at most $3$.
            On the other hand, by \eqref{p1} and \eqref{p2}, the first return maps $P_1$ and $P_2$ are hyperbolic and, therefore, the intersections of any spiral with a transversal to the polycycle has Minkowski dimension $0$ (see \cite[Lemma 1]{neveda07}).
            By Corollary~\ref{cor1} we conclude that $d = 1$.

            Consider now the case when $\beta_{1,2}\beta_{2,1}^{r_1} = 1$. By \eqref{p1} and \eqref{p2}, it follows that $|\alpha_1| + |\alpha_2| \neq 0$. Indeed, if $|\alpha_1| + |\alpha_2| = 0$, the first return maps are equal to the identity and the polycycle is trivial (of center type), which is a contradiction with the assumption. Therefore, in \eqref{p1} and \eqref{p2}, at least one of $k_1$ and $k_2$ is finite. Without loss of generality (see the discussion at the beginning of the section) we assume that
            \begin{equation}\label{eq:cas}
                k_1 - 1 < r_1(k_2-1).
            \end{equation}
            The Minkowski dimension of an orbit of $P_1$ is $1 - \frac{1}{k_1}$ and the Minkowski dimension of an orbit of $P_2$ is  $ 1 - \frac{1}{r_2(k_1-1)+1}$  (see \cite[Theorem 1]{neveda07}). Now we distinguish two cases: 
            \begin{enumerate}

            \item $r_1 > 1$ 
            
            Since $r_1r_2=1$, it follows that $r_2<1$, so $r_2(k_1-1) < k_1-1$.
            By Corollary~\ref{cor1} we conclude that $d = 2 - \frac{1}{k_1} \in \mathbb{Q}$.
            
            By Theorem~\ref{mrt3}, cyclicity in case \eqref{eq:cas} is at most
            \[ 2 + k_1 + \left\lfloor \frac{k_1-1}{r_1} \right\rfloor = \left\lfloor 2 + k_1 + \frac{k_1-1}{r_1} \right\rfloor = \left\lfloor 3 + \frac{d-1}{2-d} + r_2\frac{d-1}{2-d} \right\rfloor. \] % \leq 3 + \frac{d-1}{2-d} + r_2\frac{d-1}{2-d}. \]

            \item $r_1 < 1$ 
            
            It follows that $r_2>1$, so $r_2(k_1-1) > k_1-1.$ Similarly as in the previous case we get that $d = 2 - \frac{1}{1 + r_2(k_1-1)} \not\in \mathbb{Q}$.
            The cyclicity of the polycycle is at most
            \[ 2 + k_1 + \left\lfloor \frac{k_1-1}{r_1} \right\rfloor = \left\lfloor 2 + k_1 + \frac{k_1-1}{r_1} \right\rfloor = \left\lfloor 3 + r_1\frac{d-1}{2-d} + \frac{d-1}{2-d} \right\rfloor. \] % \leq 3 + \frac{d-1}{2-d} + r_2\frac{d-1}{2-d}. \]
         \end{enumerate}
         In the symmetrical case $k_2-1<r_2(k_1-1)$ in \eqref{eq:cas} similar conclusions hold, and the statement of the theorem follows.
        \end{proof}
 \smallskip

        The following corollary unites previous results in non-degenerate and degenerate $\mathcal{C}_1$ and $\mathcal{C}_2$ cases.

        \begin{corollary}\label{obs:a}
        Let $\Gamma_2$ be a non-trivial $2$-saddle polycycle of an analytic vector field such that $\Gamma_2\notin\mathcal C_3$. Let $S$ be its one accumulating spiral trajectory. 
        \begin{enumerate}
        \item If $\dim_B S=1$, then the cyclicity is at most $3$.
        \item If $d:=\dim_B S\in(1,2)$, then the upper bound on cyclicity is given by formula \eqref{cycl} of Theorem~\ref{thm4}, and
        $$
        r=\min\left\{\frac{d_1-1}{d_2-1},\frac{d_2-1}{d_1-1}\right\},
        $$
        $d_1\in(0,1)$ and $d_2\in(0,1)$ Minkowski dimensions of sequences obtained as intersections of spiral $S$ with transversals to the two heteroclinic connections. Note also that $d=1+\max\{d_1,d_2\}$.
        \end{enumerate}
        \end{corollary}
        \begin{proof} By Theorems~\ref{thm5}, \ref{thm6} and \ref{thm4}, $\dim_B S=1$ for non-degenerate cycles, for family $\mathcal C_1$ and in the case $r_1r_2\neq 1$ in $\mathcal{C}_2$. In all those cases the first return maps are strongly hyperbolic or hyperbolic. In all these families, by Theorems~\ref{teo0}, \ref{teo1} and \ref{mrt3} of Mourtada, the upper bound on cyclicity is $3$. On the other hand, if $r_1\cdot r_2=1$ and $r_1,\,r_2\notin\mathbb Q$, the first return maps on both transversals are tangent to the identity, with multiplicities $\gamma_1$ and $\gamma_2$ striclty greater than $1$. It is easy to check by e.g. \cite{neveda07} that $r_1=\frac{\gamma_1}{\gamma_2}$ and $r_2=\frac{\gamma_2}{\gamma_1}$. By \cite{neveda07}, $d_1=1-\frac{1}{\gamma_1}$ and $d_2=1-\frac{1}{\gamma_2}$, and the above formula for $r:=\min\{r_1,r_2\}$ follows.
        \end{proof}

\begin{remark}
    Note that a trivial saddle polycycle is of center type (no spiral trajectories). The first return map on transversals to heteroclinic connections is equal to the identity. The Minkowski dimension of just one periodic trajectory close to the polycycle is $1$ (moreover, of finite length). On the other hand, the continuum of periodic trajectories accumulating on the polycycle is an open set of non-zero area and its Minkowski dimension is equal to $2$. None of the two makes much sense to consider. Therefore, we exclude this case from our fractal considerations. 
    
    In all non-trivial cases, by quasi-analyticity of first return maps around hyperbolic polycycles of planar analytic vector fields, the first return map on transversals is never the identity, but either tangent to the identity or (strongly) hyperbolic. Therefore its orbits on transversals to heteroclinic connections have Minkowski dimension belonging to $[0,1)$, by e.g. \cite{neveda07}. By Corollary~\ref{cor1}, the Minkowski dimension of a spiral trajectory $S$ around the non-trivial hyperbolic saddle polycycle then satisfies $\dim_B S\in [1,2)$. 
\end{remark}
     
    %\textbf{note to self:} U sva tri slučaja vrijedi da je gornja ograda jednaka $3 + \frac{d-1}{2-d} + r\frac{d-1}{2-d}$ gdje je $1 > r \in \{r_1, r_2\}$. Možda samo tako iskazati teorem, a onda te slučajeve razlikovati samo u dokazu?
\section*{Acknowledgements}

This research was supported by Croatian Science Foundation (HRZZ) Grant PZS-2019-02-3055 from Research Cooperability program funded by the European Social Fund. The first two authors are supported by the Special Research Fund (BOF number: BOF21BL01) of Hasselt University. The third author is supported by Croatian Science Foundation (HRZZ) Grant UIP-2017-05-1020. The first and the third author are also supported by the bilateral Hubert-Curien Cogito grant 2023-24.

\bibliographystyle{plain}

\begin{thebibliography}{10}

\bibitem{dumllibart06}
F.~Dumortier, J.~Llibre, and J.~C. Art{\'e}s.
\newblock {\em Qualitative theory of planar differential systems}.
\newblock Universitext. Berlin: Springer, 2006.

\bibitem{neveda07}
N.~Elezovi\' c, V.~\v Zupanovi\' c, and D.~\v Zubrini\' c.
\newblock Box dimension of trajectories of some discrete dynamical systems.
\newblock {\em Chaos, Solitons \& Fractals}, 34(2):244--252, 2007.

\bibitem{falconer90}
K.~Falconer.
\newblock {\em Fractal geometry - mathematical foundations and applications.}
\newblock Wiley, 1990.

\bibitem{ilya91}
Yu.~S. Il'yashenko.
\newblock {\em Finiteness theorems for limit cycles. {Transl}. from the
  {Russian} by {H}. {H}. {McFaden}}, volume~94 of {\em Transl. Math. Monogr.}
\newblock Providence, RI: American Mathematical Society, 1991.

\bibitem{MarResZup12}
P.~Marde{\v{s}}i{\'c}, M.~Resman, and V.~{\v{Z}}upanovi{\'c}.
\newblock Multiplicity of fixed points and growth of {{\(\varepsilon
  \)}}-neighborhoods of orbits.
\newblock {\em J. Differ. Equations}, 253(8):2493--2514, 2012.

\bibitem{mourtada94}
A.~Mourtada.
\newblock Degenerate and non-trivial hyperbolic polycycles with two vertices.
\newblock {\em Journal of Differential Equations}, 113(1):68--83, 1994.

\bibitem{roussarie98}
R.~Roussarie.
\newblock {\em Bifurcation of planar vector fields and {Hilbert}'s sixteenth
  problem}, volume 164 of {\em Prog. Math.}
\newblock Basel: Birkh{\"a}user, 1998.

\bibitem{tricot95}
C.~Tricot.
\newblock {\em Curves and Fractal Dimension}.
\newblock Springer New York, New York, NY, 1995.

\bibitem{zubzup08}
D.~{\v{Z}}ubrini{\'c} and V.~{\v{Z}}upanovi{\'c}.
\newblock Poincar{\'e} map in fractal analysis of spiral trajectories of planar
  vector fields.
\newblock {\em Bull. Belg. Math. Soc. - Simon Stevin}, 15(5):947--960, 2008.

\bibitem{zubzup05}
D.~\v Zubrini\' c and V.~\v Zupanovi\' c.
\newblock Fractal analysis of spiral trajectories of some planar vector fields.
\newblock {\em Bulletin des Sciences Mathématiques}, 129(6):457--485, 2005.

\end{thebibliography}

\end{document}